\documentclass[a4paper,12pt]{article}
\usepackage{a4wide}
\usepackage{amsmath}
\usepackage{amssymb}
\usepackage{amsthm}
\usepackage{latexsym}
\usepackage{graphicx}
\usepackage[english]{babel}
\usepackage{makeidx}
\usepackage{mathtools}
              
\newtheorem{dfn} [subsection]{Definition}
\newtheorem{obs} [subsection]{Remark}

\newtheorem{prop}[subsection]{Proposition}

\newtheorem{teor}[subsection]{Theorem}
\newtheorem{lema}[subsection]{Lemma}

\def\supp{\operatorname{Cons}}
\def\Cons{\operatorname{Cons}}

\def\SS{\operatorname{S}}
\def\CC{\operatorname{C}}
\def\AA{\operatorname{A}}
\def\SL{\operatorname{SL}}
\def\GL{\operatorname{GL}}
\def\Sp{\operatorname{Sp}}
\def\Ar{\operatorname{Ar}}
\def\Ch{\operatorname{Ch}}
\def\eM{\operatorname{M}}

\def\Gal{\operatorname{Gal}}
\def\Hol{\operatorname{Hol}}
\def\Lin{\operatorname{Lin}}

\def\Irr{\operatorname{Irr}}

\def\Ker{\operatorname{Ker}}
\numberwithin{equation}{section}
\usepackage{titletoc}
\dottedcontents{section}[2.5em]{}{2em}{1pc}
\begin{document}
\selectlanguage{english}
\frenchspacing

\title{On a generalization of monomial groups}
\author{Mircea Cimpoea\c s$^1$}
\date{}

\maketitle

\begin{abstract}
We study a class of finite groups, called almost monomial groups,
which generalize the class of monomial groups and is connected with the theory 
of Artin L-functions. Our method of research is based on finding similarities with the theory
of monomial groups, whenever it is possible.


\end{abstract}

\footnotetext[1]{ \emph{Mircea Cimpoea\c s}, University Politehnica of Bucharest, Faculty of
Applied Sciences, 
Bucharest, 060042, Romania and Simion Stoilow Institute of Mathematics, Research unit 5, P.O.Box 1-764,
Bucharest 014700, Romania, E-mail: mircea.cimpoeas@upb.ro,\;mircea.cimpoeas@imar.ro}

\section*{Introduction}

The notion of \emph{almost monomial groups}, which is a loose generalization of monomial groups,
was introduced by F.\ Nicolae in a recent paper \cite{monat}. 
A finite group $G$ is called \emph{almost monomial}, if for any two irreducible characters $\chi\neq \phi$ of $G$,
there exist a subgroup $H\leqslant G$ and a linear character $\lambda$ of $H$, such that $\chi$ is a constituent
of the induced character $\lambda^G$ and $\phi$ is not. Some basic properties of almost monomial groups
are presented in \cite{lucrare}. The aim of this paper is to continue the study of almost monomial groups. 
We mention that a previous non related notion of "almost monomial groups" appeared in \cite{book}.

In the first section, we recall some basic definitions and properties from the character theory of finite groups.
Also, we present the connections between almost monomial groups and the theory of Artin L-functions,
which motivates our study. In the second section, we prove some equivalent 
characterizations of almost monomial groups, see Theorem \ref{31}.

A natural way to study the almost monomial groups is to find properties similar with those of monomial groups. For instance,
according to a theorem of Dornhoff \cite{dorn}, if $G$ is monomial and $N\unlhd G$ is a normal Hall subgroup, then $N$ is also
monomial. A similar property is no longer valid for almost monomial groups; an example is provided in the Section $3$. Using
Clifford's theory, we prove that if $G$ is almost monomial and $N\unlhd G$ that satisfies certain technical conditions, then 
$N$ is also almost monomial; see Theorem \ref{43}.

In the fourth section we introduce the notion of \emph{relative almost monomial} groups, similar to relative monomial groups (see \cite[p. 86]{isaac}), and we 
prove several results. If $G$ is a group and $N\unlhd G$ is a normal subgroup, we say that $G$ is \emph{relative almost monomial}
with respect to $N$ if for any irreducible characters $\chi\neq \phi$ of $G$, there exists a subgroup $H\leqslant G$ with
 $N\subseteq H\subseteq G$ and an irreducible character $\psi$ of $H$, such that the restriction $\psi_N$ is irreducible,
$\chi$ is a constituent of $\psi^G$ and $\phi$ is not a constituent of $\psi^G$.

In Proposition \ref{52} we note that a group 
$G$ is almost monomial if and only if it is relative almost monomial w.r.t. the trivial subgroup and, also, we prove that if
$G$ is relative almost monomial w.r.t. $N\unlhd G$, then $G/N$ is almost monomial. In Theorem \ref{53} we
prove that if $G$ is solvable and relative almost monomial w.r.t. $N\unlhd G$ and all the Sylow subgroups of $N$ are abelian,
then $G$ is almost monomial. In Theorem \ref{54} we prove that $G_1,G_2$ are almost monomial w.r.t. $N_1, N_2$, if and only if
$G_1\times G_2$ is almost monomial w.r.t. $N_1\times N_2$. In Theorem \ref{57} we prove that if $G$ is relative almost monomial w.r.t. 
$N\unlhd G$, and $A\unlhd G$ with $A\subseteq N$, then $G/A$ is relative almost monomial w.r.t. $N/A$. In Theorem \ref{58} we prove that
if $G$ is almost monomial and $A\unlhd G$ is abelian, then $G$ is relative almost monomial w.r.t. $A$.

In his thesis \cite{guan}, Guan Aun How investigated nM-groups (sM-groups): finite groups whose
irreducible  characters are all induced from linear characters of normal (subnormal) subgroups.
In Section $5$, we introduce their counterparts in the frame of almost monomial groups, namely normal almost monomial groups 
and subnormal almost monomial groups and we show that these classes of almost monomial groups are
closed under taking factor and direct products, see Theorem \ref{62}. Also, we give examples of groups which are subnormally
almost monomial, but are not normally almost monomial, and groups which are almost monomial, but are not subnormally almost monomial.


In Section $6$ we present our functions in GAP \cite{gap} which determine if a group $G$ is almost monomial, normally almost monomial, 
subnormally almost monomial or almost monomial with respect to a normal subgroup $N\unlhd G$. 

\section{Preliminaries and motivation}

Let $G$ be a finite group. We denote $\Ch(G)$ the set of characters associated to the linear 
representations of the group $G$ over the complex field. We denote $\Irr(G)=\{\chi_1,\ldots,\chi_r\}$ 
the set of irreducible characters of $G$. It is well known that any character $\chi$ of $G$ can be uniquely written 
as linear combination $\chi=a_1\chi_1+a_2\chi_2+\cdots+a_r\chi_r$ where $a_i$'s are nonnegative integers and
not all of them are zero. A character $\lambda$ of $G$ is called linear, if $\lambda(1)=1$. Obviously, the linear
characters are irreducible.

If $H\leqslant G$ is a subgroup and $\chi$ is a character of $G$, then the restriction of $\chi$ to $H$, denoted by $\chi_H$,
is a character of $H$. If $\theta$ is a character of $H$, then 
$$\theta^G(g):= \frac{1}{|H|}\sum_{x\in G} \theta^0(xgx^{-1}),\;\text{ for all }g\in G,$$
where $\theta^0(x)=\theta(x)$, for all $x\in H$, and $\theta^0(x)=0$, for all $x\in G\setminus H$, is a character of $G$, which is
called the character induced by $\theta$ on $G$.

A character $\chi\in \Ch(G)$ is called \emph{monomial} if there exist a subgroup $H \leqslant G$ and a linear character
$\lambda$ of $H$ such that $\chi=\lambda^G$. A group $G$ is called \emph{monomial}, or $M$-group, if all the irreducible characters of 
$G$ are monomial. We mention that, according to a theorem of Taketa, see \cite{taketa} or \cite[Theorem A]{dorn}, all monomial groups are solvable.
In general, the converse is not true, the smallest example being the group $\SL_2(\mathbb F_3)$; see \cite[p. 67]{isaac2}.

A character $\chi\in \Ch(G)$ is called \emph{quasi-monomial} if there exist a subgroup $H \leqslant G$, a linear character
$\lambda$ of $H$ and an integer $d\geq 1$ such that $\lambda^G = d\cdot \chi$. 
A group $G$ is called \emph{quasi-monomial} if any irreducible character of $G$ is quasi-monomial.
It is not known if there are quasi-monomial groups which are not monomial. However, it was proved in \cite{konig} that a special class of 
quasi-monomial groups are solvable.

It is well known that 
if
$\chi\in\Ch(G)$, then $\chi = \langle \chi_1,\chi \rangle \chi_1 + \cdots + \langle \chi_r,\chi \rangle \chi_r$. Hence, the 
\emph{set of constituents} of $\chi$ is
$\supp(\chi):=\{\phi \in \Irr(G)\;:\; \langle \phi,\chi \rangle > 0\}$.
We recall the following definition from \cite{monat}; see also \cite{lucrare}:

\begin{dfn}
A group $G$ is called \emph{almost monomial} (or \emph{$AM$-group}), if for any $\chi\neq \phi \in \Irr(G)$ 
there exists a subgroup $H\leqslant G$ and a linear character $\lambda$ of $H$ such that $\chi\in\supp(\lambda^G)$ and $\phi\notin \supp(\lambda^G)$.
\end{dfn}

Obviously, any (quasi)-monomial group $G$ is also almost monomial; the converse however is false. For instance, the already mentioned group $\SL_2(\mathbb F_3)$
is almost monomial, but is not monomial. We recall the main results from \cite{lucrare}, regarding the almost monomial groups:

\begin{teor}\label{22} We have that:
\begin{enumerate}
\item[(1)] The symmetric group $\SS_n$ is almost monomial for any $n\geq 1$.(\cite[Theorem 2.1]{lucrare})
\item[(2)] If $G$ is almost monomial and $N\unlhd G$ is a normal subgroup, then $G/N$ is almost monomial.(\cite[Theorem 2.2]{lucrare})
\item[(3)] $G_1$ and $G_2$ are almost monomial if and only if $G_1\times G_2$ is almost monomial. (\cite[Theorem 2.3]{lucrare})
\end{enumerate}
\end{teor}

\begin{obs}\emph{
We recall that a group $G$ is called \emph{rational} if $\chi(g)\in\mathbb Q$ for any character $\chi$ of $G$ and any $g\in G$.
It is well known that the symmetric groups $\SS_n$ are rational; see \cite{james}. We may ask if all rational groups are almost monomial. 
This question has a negative answer. The projective symplectic group $\Sp_6(\mathbb F_2)$ is a simple rational group, see \cite{feit}, but is not 
almost monomial, according to our computations in GAP \cite{gap}. On the other hand, any finite Weyl group is rational, see \cite[Corollary 1.14]{spring}. 
Our computer experiments yield us to conjecture that the finite Weyl groups of the type $A_n$, $B_n$, $C_n$ and $D_n$ are almost monomial.}
\end{obs}

The main motivation in studying almost monomial groups is given by their connection with the theory of Artin L-functions.
Let $K/\mathbb Q$ be a finite Galois extension. 
For the character $\chi$ of a representation of the Galois group $G:=Gal(K/\mathbb Q)$
on a finite dimensional complex vector space, let $L(s,\chi):=L(s,\chi,K/\mathbb Q)$ be the corresponding Artin L-function 
(\cite[P.296]{artin2}). 
Artin conjectured that $L(s,\chi)$ is holomorphic in $\mathbb C\setminus \{1\}$ and $s=1$ is a simple pole. 
Brauer \cite{brauer} proved that $L(s,\chi)$ is meromorphic in $\mathbb C$, of order $1$.

\begin{obs}\emph{
Let $f\in \mathbb Q[X]$ be a generic polynomial of degree $n\geq 2$. Then the Galois group 
$\Gal(\mathbb Q(f)/\mathbb Q)$ is isomorphic to $\SS_n$, which, according to Theorem \ref{22}(1), is an almost monomial group.}
\end{obs}

Let $\chi_1,\ldots,\chi_r$ be the irreducible characters of $G$, $f_1=L(s, \chi_1),\ldots,f_r=L(s,\chi_r)$ the corresponding Artin L-functions.
In \cite{forum} we proved that $f_1,\ldots,f_r$ are algebraically independent over the field of meromorphic functions of order $<1$.
We consider
$$\Ar:=\{f_1^{k_1}\cdot\ldots\cdot f_r^{k_r}\mid k_1\geq 0,\ldots,k_r\geq 0\}$$
the multiplicative semigroup of all  L-functions. For $s_0\in\mathbb C,s_0\neq 1$ let $\Hol(s_0)$ be the
subsemigroup of $\Ar$ consisting of the L-functions which are holomorphic at $s_0$. 
It is well known, that if $\psi$ is a quasi-monomial character of $G$, then $L(s,\psi)$ is holomorphic 
on $\mathbb C\setminus \{1\}$; see for instance \cite[Lemma 1.2]{konig}. As a direct consequence, it 
follows that if $G$ a quasi-monomial, then Artin's conjecture holds for $G$.

The main result of \cite{monat} is the following:

\begin{teor}\label{25}
If $G$ is almost monomial, then the following are equivalent:
\begin{enumerate}
\item[(1)] Artin's conjecture is true: $\Hol(s_0)=\Ar.$
\item[(2)] The semigroup $\Hol(s_0)$ is factorial. 
\end{enumerate}
\end{teor}

Our main result of \cite{lucrare} is the following:

\begin{teor}\label{26}
If $G$ is almost monomial and $s_0$ is not a common zero for any two distinct L-functions $f_k$ and $f_l$ then all Artin L-functions of $K/\mathbb Q$ are holomorphic at $s_0$.  
\end{teor}

\section{A combinatorial characterization}

Let $G$ be a finite group with $\Irr(G)=\{\chi_1,\ldots,\chi_r\}$. We consider $\eM(G) \subseteq \Ch(G)$, the subsemigroup generated by monomial characters, i.e.
$\psi\in \eM(G)$ if there exist some subgroups $H_1,\ldots,H_k$ of $G$ and some linear characters $\lambda_i$ of $H_i$, $1\leq i\leq k$,
such that $\psi=\lambda_1^G+\cdots+\lambda_k^G$.

For $1\leq t\leq  r$, we consider the numbers
\begin{equation*}
 L_t(G):=|\{\supp(\chi)\;:\;\chi\in \eM(G)\text{ and }|\supp(\chi)|=t\}|.
\end{equation*}
Obviously, $L_t(G)\leq \binom{r}{t}$. Also,
note that $L_1(G)=r$ if and only if the group $G$ is quasi-monomial.
For any group $G$, the \emph{regular character} is
\begin{equation*}\label{roge}
\rho_G:=d_1\chi_1+\cdots+d_r\chi_r = 1_{\{1\}}^G,\text{ where } d_i=\chi_i(1),\; 1\leq i\leq r.
\end{equation*}
It follows that $L_r(G)=1$. We consider the numbers: 
$$N_{r,t}:=\binom{r}{t} - \binom{r-2}{t-1} + 1,\;\text{for all }1\leq t\leq r-1.$$
We prove the following result, which is similar to \cite[Theorem 1.6]{cim}:

\begin{teor}\label{31}
 The following are equivalent:
\begin{enumerate}
 \item[(1)] $G$ is almost monomial.
 \item[(2)] There exists $1\leq t\leq r-1$ such that $L_t(G)\geq N_{r,t}$.
 \item[(3)] $L_{r-1}(G)=r$.
\end{enumerate}
\end{teor}

\begin{proof}
$(1)\Rightarrow (3)$ We fix $1\leq i\leq r$. Since $G$ is almost monomial, for any $1\leq j\leq r$ with $j\neq i$, there exists 
a subgroup $H_j\leqslant G$ and a linear character $\lambda_j$ of $H_j$ such that $\chi_j\in\supp(\lambda_j^G)$ and 
$\chi_i\notin\supp(\lambda_j^G)$.
Let $\psi:=\sum_{j\neq i}\lambda_j^G$. We have that $\supp(\psi)=\{\chi_1,\ldots,\chi_r\}\setminus\{\chi_i\}$.
Since $i$ was arbitrary chosen, it follows that $L_{r-1}(G)=r$. 

$(3)\Rightarrow (2)$ It is obvious, as $N_{r,r-1}=r$.

$(2)\Rightarrow (1)$ We fix $1\leq t\leq r-1$ such that $L_t(G)\geq N_{r,t}$. Assume, by contradiction,
that there exists $i\neq j$ such that for any $\psi\in \eM(G)$, 
$\supp(\psi)\cap \{\chi_i,\chi_j\}\neq \{\chi_i\}$. 

We choose a subset $A\subseteq \{\chi_1,\ldots,\chi_r\}$ 
with $t$ elements such that $\chi_i\in A$ and $\chi_j\notin A$. It follows that
$A\setminus \{\chi_i\}$ is a subset with $t-1$ elements in $\{\chi_1,\ldots,\chi_r\}\setminus \{\chi_i,\chi_j\}$,
hence $A$ can be chosen in $\binom{r-2}{t-1}$ ways. 
Therefore $L_t(G)\leq \binom{r}{t} - \binom{r-2}{t-1} = N_t-1$, a contradiction.
\end{proof}

\section{Normal subgroups of almost monomial groups}

Let $G$ be a finite group and $N\unlhd G$ a normal subgroup. In general, if $G$ is monomial, then $N$ is not necessarily monomial,
an example being provided independently by E. C. Dade \cite{dade} and R. van der Waall \cite{waall}, with order $2^9\cdot 7$. 
A similar fact is true in the almost monomial case. For example, the group $\SS_6$ is almost monomial, but $N:=\AA_6\unlhd \SS_6$ is not almost monomial.

On the other hand, if $N\unlhd G$ is a normal Hall subgroup, that is $(|G:N|,|N|)=1$, Dornhoff \cite{dorn} proved that if $G$ is monomial 
then $N$ is also monomial. It is natural to ask the following question: If $G$ is almost monomial and $N\unlhd G$ is Hall, is then $N$ almost 
monomial? Our computer experiments in GAP \cite{gap} show that this statement is false, in general; see Section $6$.
In order to construct such examples, we search for a group $N$ of odd order, that is not almost monomial, and we extend it to a group $G$ with 
$[G:N]=2$, which might be almost monomial.
We consider the following list of finite groups from the Small Groups library of GAP \cite{gap}, 
with the property that they have odd orders and are not normally monomial, see \cite[Page 104]{punjab}:
\begin{align*}
& \text{SmallGroup(375,2)},\; \text{SmallGroup(1029,12)},\; \text{SmallGroup(1053,51)},\; \text{SmallGroup(1125,3)},\\
& \text{SmallGroup(1125,7)},\; \text{SmallGroup(1215,68)}, \;\text{SmallGroup(1875,18)},\; \text{SmallGroup(1875,19)}, 
\end{align*}
All the groups in the above list are not almost monomial, with the exceptions of the second and the third, which are monomial.
We choose $N$ to be one of the following groups:  SmallGroup(375,2), SmallGroup(1125,3), SmallGroup(1215,68), SmallGroup(1875,18), SmallGroup(1875,19).
Then, we can find $G:=N \rtimes \CC_2$ a non-trivial semidirect product of $N$ with $\CC_2=$ the cyclic group of order $2$, such that $G$ is almost monomial. 

In the following, given $G$ an almost monomial group and $N\unlhd G$ a normal subgroup, we will 
give a sufficient condition for $N$ to be almost monomial. But first, we need two lemmas:

\begin{lema}(\cite[Problem (5.2)]{isaac})\label{41}
If $N \unlhd G$, $H\leqslant G$ and $\varphi$ is a character of $H$, 
then $$(\varphi^{NH})_N = (\varphi_{N\cap H})^N.$$
\end{lema}

\begin{lema}(\cite[Corollary 1.16]{isaac2})\label{42}
Let $N \unlhd G$, and suppose that $\chi \in \Irr(G)$ has degree relatively prime to $|G : N|$. Then $\chi_N$ is irreducible.
\end{lema}

\begin{teor}\label{43}
Let $N \unlhd G$ be a normal subgroup of $G$. Assume that:
\begin{enumerate}
 \item[(i)] Every irreducible character $\chi$ of $G$ restricts irreducibly to $N$.
 \item[(ii)] If $H\leqslant G$ is a subgroup, $\lambda$ is a linear character of $H$ and $\phi$ is a irreducible character of $G$ such that 
$\langle \lambda^G, \phi \rangle = 0$, then $\langle \lambda^{NH}, (\phi_N)^{NH} \rangle = 0$.
\end{enumerate}
If $G$ is almost monomial, then $N$ is almost monomial.
\end{teor}

\begin{proof}
Let $\theta\neq \eta\in \Irr(N)$, $\chi\in\Cons(\theta^G)$ and $\phi\in \Cons(\eta^G)$. By $(i)$, we may assume that $\chi\neq \phi$.
Since $G$ is almost monomial, there exists a subgroup $H\leqslant G$ and a linear character $\lambda$ of $H$ such that $\langle \lambda^G,\chi \rangle \neq 0$
and $\langle \lambda^G,\phi \rangle = 0$. We consider the subgroup $H\cap N\leqslant N$ and its linear character $\lambda_{H\cap N}$.
By Frobenius reciprocity and Lemma \ref{41} it follows that
$$ \langle (\lambda_{H\cap N})^N, \theta \rangle = \langle (\lambda^{HN})_N, \theta \rangle =  
\langle \lambda^{HN}, \theta^{HN} \rangle \geq \langle \lambda^{HN}, \chi_{HN} \rangle =  \langle (\lambda^{HN})^G, \chi \rangle = \langle \lambda^G, \chi \rangle > 0.$$
Since $\langle \lambda^G, \phi \rangle = 0$, by $(ii)$ and Lemma \ref{41}, it follows that 
$$0=\langle \lambda^{NH}, (\phi_N)^{NH} \rangle = \langle (\lambda^{NH})_N, \eta \rangle = \langle (\lambda_{H\cap N})^N, \eta \rangle.$$ 
Thus $N$ is almost monomial.
\end{proof}

\begin{obs}
\emph{
According to Lemma \ref{42}, the condition $(i)$ of Theorem \ref{43} is implied by the assertion: 
Every irreducible character $\chi$ of $G$ has degree relatively prime with $[G:N]$. The condition $(ii)$ is more technical,
but it is necessary in the proof.}
\end{obs}

\section{Relative almost monomial groups}

Let $G$ be a finite group and $N\unlhd G$ a normal subgroup. We recall that the group $G$ is called \emph{relative monomial} with respect to $N$, if for any $\chi\in \Irr(G)$,
there exists a subgroup $H\leqslant G$ with $N\subseteq H$ and an irreducible character $\psi$ of $H$ such that $\psi_N$ is irreducible and $\psi^G=\chi$; see \cite[Definition 6.21]{isaac}.
We introduce the following similar definition, in the framework of almost monomial groups:

\begin{dfn}\label{51}
Let $N \unlhd G$ be a normal subgroup of $G$. We say that $G$ is \emph{relative almost monomial} with respect to $N$, if for any 
$\chi\neq \phi \in \Irr(G)$ there exists a subgroup $H\leqslant G$ with $N\subseteq H\subseteq G$ and $\psi\in \Irr(H)$ such that 
$\psi_N \in \Irr(N)$, $\chi\in\supp(\psi^G)$ and $\phi\notin \supp(\psi^G)$.
\end{dfn}


\begin{prop}\label{52}
Let $G$ be a group and $N \unlhd G$ a normal subgroup. We have that:
\begin{enumerate}
 \item[(1)] $G$ is almost monomial if and only if $G$ is relative almost monomial with respect to the 
trivial subgroup $\{1\} \unlhd G$.
 \item[(2)] If $G$ is relative almost monomial with respect to $N$, then $G/N$ is almost monomial.
\end{enumerate}
\end{prop}

\begin{proof}
(1) Let $\psi\in \Irr(H)$. It is enough to note that $\psi_{\{1\}}$ is irreducible if and only if $\psi(1)=1$, that is, $\psi$ is linear.

(2) Let $\tilde \chi\neq \tilde \phi$ be two irreducible characters of $G/N$ and let $\chi$ and $\phi$ be their corresponding characters of $G$.
Obviously, $N\subseteq \Ker(\chi)$ and $N\subset\Ker(\phi)$. Since $G$ is a relative almost monomial with respect to $N$, it follows that there exists
a subgroup $H\leqslant G$ with $N\subseteq H\subseteq G$ and $\psi\in \Irr(H)$ such that $\psi_N \in \Irr(N)$, 
$\chi\in\supp(\psi^G)$ and $\phi\notin \supp(\psi^G)$.

We claim that $N\subseteq \Ker(\psi)$. Indeed, as
    $$0 < \langle \chi, \psi^G \rangle = \langle \chi_H, \psi \rangle
       \leq \langle \chi_H, (\psi_N)^H \rangle
       = \langle \chi_N, \psi_N \rangle,$$
    and since $\chi_N$ is a multiple of $1_N$ and $\psi_N \in \Irr(N)$,
    we get $\psi_N = 1_N$.
Therefore, $\psi$ is linear. Considering the corresponding linear character $\tilde\psi$ of $H/N$, the definition of almost monomial groups is satisfied by $G/N$.
\end{proof}

According to Theorem \ref{22}(2), if $G$ is almost monomial and $N\unlhd G$, then $G/N$ is almost monomial. The converse
is not true, even when $N$ and $G/N$ are both almost monomial. For example, $G:=\GL_2(\mathbb F_3)$
is not almost monomial, but $N:=\SL_2(\mathbb F_3)$ and $G/N\cong \CC_2$ are. 

\begin{teor}\label{53}
 Let $N \unlhd G$ be a normal subgroup of $G$ and assume that all Sylow subgroups of $N$ are abelian. Assume that $G$ is solvable
and is a relative almost monomial group w.r.t. $N$. Then $G$ is almost monomial.
\end{teor}

\begin{proof}
Let $\chi\neq \phi\in\Irr(G)$ and choose $H\leqslant G$ with $N\subseteq H\subseteq G$, $\psi\in \Irr(H)$ such that $\psi_N \in \Irr(N)$, 
$\chi\in\supp(\psi^G)$ and $\phi\notin \supp(\psi^G)$. We choose a subgroup $U\subseteq H$ minimal such that
there exists $\theta\in \Irr(U)$ with $\theta^H = \psi$. From the proof of \cite[Theorem 6.23]{isaac}, 
it follows that $\theta\in \Lin(U)$. Since $\theta^G = (\theta^H)^G = \psi^G$, we get the required conclusion.
\end{proof}

We recall the following well known fact:

\begin{lema}\label{54}(\cite[Theorem 4.21]{isaac})
 Let $G=H\times K$. Then the characters $\varphi\times\theta$ for $\varphi\in \Irr(H)$ and $\theta\in\Irr(K)$ are
exactly the irreducible characters of $G$.
\end{lema}

The following result generalizes Theorem \ref{22}(3); compare with \cite[Proposition 1]{kwang}.

\begin{teor}\label{55}
Let $N_1\unlhd G_1$, $N_2\unlhd G_2$. The following are equivalent:
\begin{enumerate}
 \item[(1)] $G_i$ is relative almost monomial w.r.t. $N_i$, for i=$1,2$.
 \item[(2)] $G_1\times G_2$ is relative almost monomial w.r.t. $N_1\times N_2$.
\end{enumerate}
\end{teor}

\begin{proof}
$(1)\Rightarrow (2)$ Let $\chi\neq \phi \in \Irr(G_1\times G_2)$. According to Lemma \ref{54}, there exists $\chi_1,\phi_1\in \Irr(G_1)$ and 
$\chi_2,\phi_2\in \Irr(G_2)$ such that $\chi=\chi_1\times\chi_2$ and $\phi=\phi_1\times\phi_2$. 
Since $\chi\neq \phi$, it follows that $\chi_1\neq\phi_1$ or $\chi_2\neq\phi_2$. 
Without any loss of generality, assume $\chi_1\neq\phi_1$.

From hypothesis (1), there exists a subgroup $H_1\leqslant G_1$ with $N_1\subseteq H_1\subseteq G_1$ and a character $\psi_1\in \Irr(H_1)$ such that $(\psi_1)_{N_1}\in \Irr(N_1)$,
$\chi_1\in\Cons(\psi_1^G)$ and $\phi_1\notin\Cons(\psi_1^G)$. We choose $\psi_2\in\Cons((\chi_2)_{N_2})$. Let $H_2:=N_2$.
From Lemma \ref{54}, it follows that $\psi:=\psi_1\times \psi_2$ is irreducible and, moreover, 
$(\psi_1 \times \psi_2)_{N_1\times N_2} = (\psi_1)_{N_1}\times (\psi_2)_{N_2}$ is an irreducible character of $N_1\times N_2$.
One can easily check that $\chi\in\Cons(\psi^G)$ and $\phi\notin\Cons(\psi^G)$.

$(2)\Rightarrow(1)$ Let $\chi_1\neq\phi_1\in\Irr(G_1)$. We consider the characters $\chi:=\chi_1\times 1_{G_2}$ and $\phi:=\phi_1\times 1_{G_2}$. From hypothesis (2),
there exist some subgroups $H_i\leqslant G_i$ with $N_i\subseteq H_i\subseteq G_i$, $i=1,2$, and a character $\psi\in \Irr(H_1\times H_2)$ such that $\psi_{N_1\times N_2}\in \Irr(N_1\times N_2)$,
$\chi\in \Cons(\psi^G)$ and $\phi\notin \Cons(\psi^G)$. Write $\psi=\psi_1\times \psi_2$. Since $\chi= \chi_1\times 1_{G_2}\in \Cons(\psi^G)$, it follows that $1_{G_2}\in \Cons(\psi_2^{G_2})$.
Therefore, $\chi_1 \in \Cons(\psi_1^{G_1})$ and $\phi_1 \notin \Cons(\psi_1^{G_1})$, hence Definition \ref{51} is fulfilled by $G_1$ and $N_1$. Analogously, $G_2$ is
relative almost monomial w.r.t. $N_2$.
\end{proof}

We recall the following consequence of Clifford's Theorem:

\begin{lema}(\cite[Corollary 6.7]{isaac})\label{lema51}
Let $H\unlhd G$ be a normal subgroup and let $\chi\in\Irr(G)$ such that $\langle \chi_H,1_H \rangle \neq 0$. Then $H\subseteq \Ker(\chi)$.
\end{lema}

The following result generalizes Theorem $2.2(2)$; compare with \cite[Proposition 2]{kwang}.

\begin{teor}\label{57}
 Let $N\unlhd G$ such that $G$ is a relative almost monomial w.r.t. $N$. Let $A\unlhd G$ such that $A\subseteq N$. Then
$G/A$ is a relative almost monomial w.r.t. $N/A$.
\end{teor}

\begin{proof}
Let $\tilde \chi\neq\tilde \phi\in \Irr(G/A)$, and let $\chi$ and $\phi$ their corresponding characters in $\Irr(G)$. 
Since $G$ is relative almost monomial w.r.t. $N$, there exists a subgroup $H\leqslant G$ containing $N$ and a irreducible character $\psi\in \Irr(H)$ such that $\psi_N \in \Irr(N)$, $\chi\in\supp(\psi^G)$ and $\phi\notin \supp(\psi^G)$. Since $\langle \chi_H, \psi \rangle \neq 0$ and $A\subseteq \ker \chi$, it follows that
$$0 \neq \langle 1_A^H, \psi \rangle = \langle 1_A, \psi_A \rangle.$$
Hence, by Lemma \ref{lema51}, we have that $A\subseteq \Ker(\psi)$. It follows that $\psi$ defines a character $\tilde \psi$ on $G/A$,
which is irreducible on $H/A$. Moreover, $\tilde \psi_{N/A}$ is also irreducible, as $\psi_N$ is irreducible. Thus, $G/A$ is
relative almost monomial w.r.t. $N/A$.
\end{proof}

The following Theorem is similar to Problem (6.11) from \cite{isaac}.

\begin{teor}\label{58}
Let $A\unlhd G$ be abelian. If $G$ is almost monomial, then $G$ is relative almost monomial w.r.t. $A$.
\end{teor}

\begin{proof}
Let $\chi\neq\phi\in\Irr(G)$.
Since $G$ is almost monomial, there exists a subgroup $K\leqslant G$ and a linear character $\lambda$ of $K$ such that
$\chi\in\Cons(\lambda^G)$ and $\psi\notin\Cons(\lambda^G)$. Consider the subgroup $KA$ of $G$. From Lemma \ref{41} 
it follows that $(\lambda^{KA})_A=(\lambda_{K\cap A})^A$. Since $A$ is abelian, we can find a linear character $\mu$ of $A$
such that $\lambda_{K\cap A}=\mu_{K\cap A}$. Now applying \cite[Corollary 6.17]{isaac} (Gallagher) to $K\cap A \unlhd A$ and $\lambda$, it follows that
$$ (\lambda_{K\cap A})^A=(\mu_{K\cap A})^A = \sum_{\nu \in \Irr(A/K\cap A)}\mu\nu = \mu_1+\mu_2+\cdots+\mu_r,$$
and, moreover, the linear characters $\mu_i$, $1\leq i\leq r$, are distinct and are all the irreducible components of $(\lambda_{K\cap A})^A$.
Hence $(\lambda^{KA})_A$ is the sum of distinct conjugates of $\mu$, where $\mu_1=\mu$. 

Let $\theta\in\Cons(\lambda^{KA})$ such that $\chi \in \Cons(\theta^G)$. Without any loss of generality, we may assume that
$\theta_A=\mu_1+\cdots+\mu_s$, where $s\leq r$. Let $H:=I_{KA}(\mu)$ be the inertia group of $\mu$.
Applying Clifford's theorem \cite[Theorem 6.11]{isaac}, it follows that there exists an irreducible character $\psi$ of $H$
with $\langle \psi_{A},\mu\rangle >0$ such that $\psi^{KA}=\theta$. We have that
\begin{equation}\label{cucu}
 s=\theta_A(1)=\theta(1)=\psi^{KA}(1)= [KA:H]\psi(1).
\end{equation}
On the other hand, according to \cite[Theorem 6.2]{isaac} and the definition of the inertia group $I_{KA}(\mu)$, 
since $\theta_A=\mu_1+\cdots+\mu_s$, it follows that $s=[KA:H]=$ the size of the orbit of $\mu$ in $KA$.
From \eqref{cucu} it follows that $\psi(1)=1$, hence $\psi_A$ is linear. Note that $\psi_A=\mu$.
On the other hand, since $\Cons(\psi^G)\subseteq \Cons(\lambda^G)$, it follows that $\phi\notin \Cons(\psi^G)$.
\end{proof}

Note that, in the hypothesis of Theorem \ref{58}, according to Ito's theorem \cite[Theorem 6.15]{isaac}, $\chi(1)|[G:A]$ for
all $\chi\in \Irr(G)$.

\section{Two subclasses of almost monomial groups}

\begin{dfn}\label{61}
A group $G$ is called \emph{normally almost monomial} (or a \emph{$nAM$-group}), if for any $\chi \neq \phi \in \Irr(G)$, there 
exists a normal subgroup $N\unlhd G$ and a linear character $\lambda$ of $N$ such that 
$\chi\in\Cons(\lambda^G)$ and $\phi\notin\Cons(\lambda^G)$.

A group $G$ is called \emph{subnormally almost monomial} (or a \emph{$sAM$-group}), if for any $\chi \neq \phi \in \Irr(G)$, there 
exists a subnormal subgroup $H\unlhd\unlhd G$ and a linear character $\lambda$ of $H$ such that 
$\chi\in\Cons(\lambda^G)$ and $\phi\notin\Cons(\lambda^G)$.
\end{dfn}

\begin{teor}\label{62}
 The class of nAM-groups (or sAM-groups) is closed under taking factor groups and direct products.
\end{teor}

\begin{proof}
The proof is similar to the proof of \cite[Theorem 2.2]{lucrare} and \cite[Theorem 2.3]{lucrare}, taking into account the 
fact that if $N_1\unlhd G_1$ and $N_2\unlhd G_2$, then $N_1\times N_2 \unlhd G_1\times G_2$, and if $N\subseteq N_1\subseteq G$
such that $N\unlhd G$ and $N_1\unlhd G$, then $N_1/N\unlhd G/N$.
\end{proof}


\begin{obs} \emph{
Let $AM$, $nAM$ and $sAM$ be the classes of almost monomial, normally almost monomial and subnormally almost monomial groups, respectively.
Obviously, we have the inclusions $nAM \subset sAM \subset AM$. These inclusions are strict. For example, the group 
$G:=\SL_{2}(\mathbb F_3)$ is almost monomial but is not subnormally almost monomial. 
The group $G:=\text{SmallGroup}(72,40)=(S_3\times S_3)\rtimes \CC_2$ is subnormally almost monomial but is not normally almost monomial.}
\end{obs}

\section{GAP functions and computer experiments}

The following functions in GAP \cite{gap} determines if a group is almost monomial, relative almost monomial
with respect to a subgroup, normally or subnormally almost monomial.
\newline

\noindent
gap$>$  IsAlmostMonomialConditional:= function( G, subgroups, charH, charcond )\\
    local n, num, M, H, lambda, ind, constpos, j, k; \\
    n:= NrConjugacyClasses( G );\\
    num:= n * (n-1);\\
    M:= IdentityMat( n );\\
    \# \textit{Run over the allowed subgroups.}\\
    for H in subgroups do\\
    \# \textit{Run over the allowed characters of the subgroup.}\\
    for lambda in charH( H ) do\\
    ind:= InducedClassFunction( lambda, G );\\
    constpos:= PositionsProperty( Irr( G ),\\
    chi -$>$ ScalarProduct( ind, chi ) $<>$ 0 );\\
    for j in constpos do\\
    for k in Difference( [ 1 .. n ], constpos ) do\\
    \# \textit{'ind' yields a witness for the pair '(j,k)'}\\
    if M[j,k] = 0 then\\
    M[j,k]:= 1;\\
    num:= num - 1;\\
    if num = 0 then return true;fi;\\
    fi;od;od;od;od;\\
    return false;\\
end;;\\
  \\
  IsAlmostMonomial:= function( G )\\
  \;\;    return IsMonomialGroup( G ) or\\
  \;\;           IsAlmostMonomialConditional( G,\\
  \;\;\;\;               List( ConjugacyClassesSubgroups( G ), Representative ),\\
  \;\;\;\;               LinearCharacters, ReturnTrue );\\
  end;;\\
\\
  IsRelativeAlmostMonomial:= function( G, N )\\
 \;\;     local epi, subgroups;\\
 \;\;     if not ( IsSubgroup( G, N ) and IsNormal( G, N ) ) then\\
 \;\;\;\;       return fail;\\
 \;\;     fi;\\
 \;\;     epi:= NaturalHomomorphismByNormalSubgroup( G, N );\\
  \;\;    subgroups:= List( ConjugacyClassesSubgroups( Image( epi ) ),\\
  \;\;\;\;                      C -$>$ PreImage( epi, Representative( C ) ) );\\
  \;\;    return IsAlmostMonomialConditional( G,subgroups,Irr,\\
   \;\;\;\;              chi -$>$ RestrictedClassFunction( chi, N ) in Irr( N ) );\\
  end;;\\ \\
  IsNormallyAlmostMonomial:= function( G )\\
   \;\;   return IsAlmostMonomialConditional( G,NormalSubgroups( G ),\\
   \;\;\;\;              LinearCharacters,ReturnTrue );\\
  end;;\\
  IsSubnormallyAlmostMonomial:= function( G )\\
  \;\;    return IsAlmostMonomialConditional( G,\\
  \;\;\;\;               Filtered( List( ConjugacyClassesSubgroups( G ), Representative ),\\
  \;\;\;\;                         H -$>$ IsSubnormal( G, H ) ),\\
 \;\;\;\;               LinearCharacters,ReturnTrue );\\
end;;\\

We used the following code in GAP \cite{gap} to verify an example of a normal Hall subgroup $N\unlhd G$
such that $G$ is almost monomial and $N$ is not:

\noindent\\
  gap$>$ g:= SmallGroup( 750, 6 );;\\
 IsAlmostMonomial( g );\\
  true\\
 n:= Subgroup( g, GeneratorsOfGroup( g )\{[2..5]\} );;\\
  IsAlmostMonomial( n );\\
  false\\
  IdGroup( n );\\
  $[$ 375, 2 $]$\\

\textbf{Aknowledgment}. 
The author was supported by a grant of the Ministry of Research, Innovation and Digitization, CNCS - UEFISCDI, 
project number PN-III-P1-1.1-TE-2021-1633, within PNCDI III.

{}


\addcontentsline{toc}{section}{References}
\begin{thebibliography}{}
\bibitem{artin2} E.\ Artin, \emph{Zur Theorie der L-Reihen mit allgemeinen Gruppencharakteren}, Abh. Math. Sem. Hamburg \textbf{8} (1931), 292--306.
\bibitem{punjab} G.\ K.\ Bakshi, S.\ Maheshwary, \emph{Extremely strong Shoda pairs with GAP}, J. Symbolic Comput. \textbf{76} (2016), 97--106. 
\bibitem{book} A.\ R.\ Booker, \emph{Artin's conjecture, Turing's method, and the Riemann hypothesis}, 
   Experiment. Math. \textbf{15 , no. 4} (2006), 385--407.
\bibitem{brauer} R.\ Brauer, \emph{On Artin’s L-series with general group characters}, Ann. of Math. \textbf{48} (1947), 502--514.
\bibitem{cim} M.\ Cimpoea\c s, \emph{On the semigroup ring of holomorphic Artin L-functions},  Colloq. Math. \textbf{160, no. 2} (2020), 283--295.
\bibitem{forum} M.\ Cimpoea\c s, F.\ Nicolae, \emph{Independence of Artin L-functions},  Forum Math. \textbf{31, no.2} (2019), 529--534.
\bibitem{lucrare} M.\ Cimpoea\c s, F.\ Nicolae, \emph{Artin L-functions to almost monomial Galois groups},  Forum Math. \textbf{32 no. 4} (2020), 937--940.
\bibitem{dade} E.\ C.\ Dade, Normal subgroups of M-groups need not be M-groups, Math. Zeit. \textbf{133} (1973), 313--317.
\bibitem{dorn} L.\ Dornhoff, \text{M-groups and 2-groups}, Math. Zeit. \textbf{100} (1967), 226--256.
\bibitem{feit} W.\ Feit and G.\ M.\ Seitz, \emph{On finite rational groups and related topics}, Illinois J. Math, \textbf{33, No. 1}, (1988), 103--131.
\bibitem{gap}  The GAP Group, \emph{GAP Groups, Algorithms, and Programming}, Version 4.10.2 (2019), https://www.gap-system.org.
\bibitem{guan} Guan Aun How, \emph{Some classes of monomial groups}, PhD thesis, Australian National University, (1980).
\bibitem{isaac} I.\ M.\ Isaacs, \emph{Character Theory of Finite Groups}, Dover, NY, (1994).
\bibitem{isaac2} I.\ M.\ Isaacs, \emph{Characters of Solvable Groups}, Graduate Studies in Mathematics Volume 189 (2018).
\bibitem{james} G.\ James, A.\ Kerber, \emph{The Representation Theory of the Symmetric Group}, Addison-Wesley, (1991).
\bibitem{konig} J.\ K\"onig, \emph{Solvability of Generalized Monomial Groups}, J. Group Theory \textbf{13} (2010), 207--220.
\bibitem{kwang} Kwang-Wu Chen, \emph{On relative M-groups}, Chinese J. Math. \textbf{vol 21, no. 1} (1993), 1--12. 
\bibitem{monat} F.\ Nicolae, \emph{On holomorphic Artin L-functions}, Monatsh. Math. \textbf{186, no. 4}, (2018), 679--683.
\bibitem{spring} T.\ A.\ Springer, \emph{A construction of representations of Weyl groups}, Invent. Math. \emph{44 , no. 3}, (1978).
\bibitem{taketa} K.\ Taketa, \emph{\"Uber die Gruppen, deren Darstellungen sich s\"amtlich auf monomiale Gestalt transformieren lassen}, Proc. Acad. Tokyo 6 (1930), 31--33.
\bibitem{waall} R.\ W.\ van der Waall, \emph{On the embedding of minimal non-M-groups}, Indag. Math. \textbf{77} (1974), 157--167.
\end{thebibliography}
\end{document}